\documentclass[12pt]{amsart}
\usepackage{amssymb,bbold}
\usepackage[all]{xy}

\theoremstyle{plain}
\newtheorem{theorem}{Theorem}
\newtheorem{corollary}[theorem]{Corollary}

\newtheorem{proposition}[theorem]{Proposition}

\theoremstyle{definition}
\newtheorem{definition}[theorem]{Definition}
\newtheorem{remark}[theorem]{Remark}
\newtheorem{example}[theorem]{Example}

\begin{document}
\baselineskip 18pt

\title[A few remarks on boundedness]
{A few remarks on boundedness in topological modules and topological groups}
\author{Omid Zabeti}
\address{O. Zabeti, Department of Mathematics, Faculty of Mathematics, University of Sistan and Baluchestan,  P.O. Box 98135-674, Zahedan (Iran)}
\email{o.zabeti@gmail.com}

\begin{abstract}
Let X, Y, and Z be  topological modules over a topological ring $R$. In the first part of the paper, we introduce three different classes of bounded bigroup homomorphisms from $X\times Y$ into $Z$ with respect to the three different uniform convergence topologies. We show that these spaces form again topological modules over $R$. In the second part, we characterize bounded sets in the arbitrary product of topological groups with respect to the both product and box topologies.
\end{abstract}

\date{\today}

\maketitle
\section{introduction and preliminaries}
In \cite{Mi}, some notions for bounded group homomorphisms on a topological ring have been introduced. Also, it has been proved that each class of bounded group homomorphisms on a topological ring, with respect to an appropriate topology, forms a topological ring. Also, an analogous statement for topological groups and bounded homomorphisms between them has been investigated in \cite{Km}. Since every topological ring can be viewed as a topological module over itself, it is not a hard job to see that we can consider the concepts of topological modules of bounded group homomorphisms on a topological module. In fact, the results in \cite{Mi}, can be generalized to topological modules in a natural way. Recall that a \textit{topological module} $X$ is a module with a topology over a topological ring $R$ such that the addition ( as a map from $X\times X$ into $X$), and the multiplication ( as a map from $R\times X$ into $X$) are continuous. There are many examples of topological modules, for instance, every topological vector space is a topological module over a topological field, every abelian topological group is a topological module over $\Bbb Z$, where $\Bbb Z$ denotes the ring of integers with the discrete topology, and also every topological ring is a topological module over each of its subrings. So, it is of independent interest if we consider possible relations between algebraic structures of a module and its topological properties.  In the first part of the present paper, we are going to consider bounded bigroup homomorphisms between topological modules. We endow each class of bounded bigroup homomorphisms to a uniform convergence topology and we show that under the assumed topology, each class of them, forms a topological module. In addition, we see that if each class of bounded bigroup homomorphisms is uniformly complete. In the following, by a bigroup homomorphism on Cartesian product $X\times Y$, we mean a map which is group homomorphism on $X$ and $Y$, respectively. Also, note that if $X$ is a topological module over topological ring $R$, then, $B\subseteq X$ is said to be bounded if for each zero neighborhood $W\subseteq X$, there exists zero neighborhood $V\subseteq R$ such that $VB\subseteq W$. Finally, as a special case, we consider bounded sets in arbitrary Cartesian products of abelian topological groups. For more information about topological modules, topological rings, topological groups, and the related notions, see \cite{Arnautov, Taq, Km, Mi, Shell, Tr, Omid, Km1}.
\section{bounded bigroup homomorphisms}
\begin{definition}
Let $X$, $Y$, and $Z$ be topological modules over a topological ring $R$. A bigroup homomorphism $\sigma:X \times Y \to Z$ is said to be:
\begin{itemize}
\item[\em i.]{$n$-bounded if there exist some zero neighborhoods $U\subseteq X$ and $V\subseteq Y$ such that  $\sigma(U,V)$ is bounded in $Z$}.
\item[\em ii.]{${\frac{n}{2}}$-bounded if there exists a zero neighborhood $U\subseteq X$ such that for each bounded set $B\subseteq Y$, $\sigma(U,B)$ is bounded in $Z$}.
\item[\em iii.]{$b$-bounded if for every bounded subsets $ B_1\subseteq X$ and $B_2\subseteq Y$, $\sigma(B_1,B_2)$ is bounded in $Z$}.
\end{itemize}
\end{definition}
The first point is that these concepts of bounded bigroup homomorphisms are far from being equivalent. In prior to anything, we show this.
\begin{example}\label{1}
Let $X={\Bbb R}^{\Bbb N}$, the space of all real sequences, with
the coordinate-wise topology and the pointwise product. It is easy to see that $X$ is a topological module over itself. Consider
the bigroup homomorphism $\sigma:X\times X\to X$ defined by $\sigma(x,y)=xy$, in which the product is given by pointwise. It is not difficult to see that $\sigma$ is $b$-bounded but since $X$ is not locally bounded, it can not be $n$-bounded.
\end{example}
Also, the above example may apply to determine a $b$-bounded bigroup homomorphism which is not $\frac{n}{2}$-bounded.
\begin{example}
Let $X$ be $\ell_{\infty}$, the space of all bounded real sequences, with the topology induced by the uniform norm and pointwise product. Suppose $Y$ is  $\ell_{\infty}$, with the coordinate-wise topology and pointwise product. Consider the bigroup homomorphism $\sigma$ from $X \times Y$ to $Y$ as in Example \ref{1}. It is easy to see that $\sigma$ is $\frac{n}{2}$-bounded but it is not $n$-bounded. For, suppose $\varepsilon>0$ is arbitrary. Assume that $N_{\varepsilon}^{(0)}$ is the ball with centre zero and radius $\varepsilon$ in $X$. If $U$ is an arbitrary zero neighborhood in $Y$, without loss of generality, we may assume that $U$ is of the form
\[(-\varepsilon_1,\varepsilon_1)\times\ldots\times(-\varepsilon_r,\varepsilon_r)\times \Bbb R\times \Bbb R\times\ldots,\]
in which, $\varepsilon_i>0$. Fix $0<\delta<\min\{\varepsilon_i\}$. Consider the sequence $(a_n)\subseteq U$ defined by $a_n=(\delta,\ldots,\delta,n,\ldots,n,o,\ldots)$, in which $\delta$ is appeared $r$ times and $n$ equips $n-r$ components.
 Now, it is not difficult to see that $\sigma(N_{\varepsilon}^{(0)},(a_n))$ can not be a bounded subset of $Y$.
\end{example}
\begin{example}\label{200}
Let $X$ be $\ell_{\infty}$, with pointwise product and the uniform norm topology, and $Y$ be $\ell_{\infty}$, with the zero multiplication and the topology induced by norm. Consider $\sigma$ from $X\times Y$ to $X$ as in Example \ref{1}. Then, $\sigma$ is $n$-bounded but it is not $\frac{n}{2}$-bounded. For, suppose $\varepsilon>0$ is arbitrary. Consider the sequence $(a_n)$ in $Y$ defined by $a_n=(\frac{1}{\varepsilon},\ldots,\frac{n}{\varepsilon},0,\ldots)$. $(a_n)$ is bounded in $Y$ but $\sigma(N_{\varepsilon}^{(0)},(a_n))$ contains the sequence $(1,\ldots,n,0,\ldots)$ which is not bounded in $X$.
\end{example}
Since topological modules are topological spaces, we can consider the concept of jointly continuity for a bigroup homomorphism between
topological modules. The interesting result in this case, in spite of the case related to topological vector spaces and topological groups notions, is that there is no relation between jointly continuous bigroup homomorphisms and bounded ones; see \cite{Km, Omid} for more details on these concepts.
To see this, consider the following example.
\begin{example}
Let $X$ be $\ell_{\infty}$, with the pointwise product and coordinate-wise topology, and $Y$ be $\ell_{\infty}$, with the zero multiplication and the uniform norm topology. Consider the bigroup homomorphism $\sigma$ from $X\times X$ into $Y$ as in Example \ref{1}. Indeed, $\sigma$ is $b$-bounded and $n$-bounded but it is easy to see that $\sigma$ can not be jointly continuous.
\end{example}
The class of all $n$-bounded bigroup homomorphisms on a topological module $X$ is denoted by $B_{n}(X\times X)$ and is equipped with the topology of uniform
convergence on some zero neighborhoods, namely, a net $(\sigma_{\alpha})$ of $n$-bounded bigroup homomorphisms converges uniformly to zero on some zero neighborhoods
$U, V \subseteq X$ if for each zero neighborhood $W\subseteq X$ there is an $\alpha_0$
with $\sigma_{\alpha}(U,V)\subseteq W$ for each $\alpha\geq\alpha_0$.
The set of all $\frac{n}{2}$-bounded bigroup homomorphisms on a topological module $X$ is denoted by $B_{\frac{n}{2}}(X\times X)$ and it is assigned
with the topology of $\sigma$-uniformly convergence on some zero neighborhood. We say that a net $(\sigma_{\alpha})$ of $\frac{n}{2}$-bounded bigroup homomorphisms converges $\sigma$-uniformly to zero on some zero neighborhood if there exists a zero neighborhood $U\subseteq X$ such that for each bounded set $B\subseteq X$ and for each zero neighborhood $W\subseteq X$ there is an $\alpha_0$ with  $\sigma_{\alpha}(U,B)\subseteq W$ for each $\alpha\geq\alpha_0$. Finally, the class of all $b$-bounded bigroup homomorphisms on a topological module $X$ is denoted by $B_{b}(X\times X)$ and is endowed with the topology of uniform convergence on bounded sets which means a net $(\sigma_{\alpha})$ of $b$-bounded bigroup homomorphisms converges uniformly to zero on bounded sets $B_1,B_2\subseteq X$ if for each zero neighborhood $W\subseteq X$ there is an $\alpha_0$
with $\sigma_{\alpha}(B_1,B_2)\subseteq W$ for each $\alpha\geq\alpha_0$.
In this part of the paper, we show that the operations of addition and module multiplication are continuous in each of the topological modules $B_n(X\times X), B_{\frac{n}{2}}(X\times X)$, and $B_b(X\times X)$ with respect to the assumed topology, respectively. So, each of them forms a topological $R$-module.
\begin{theorem}
The operations of addition and module multiplication in $B_{n}(X\times X)$ are continuous with respect to the topology of uniform convergence on some zero neighborhoods.
\end{theorem}
\begin{proof}
Suppose two nets $(\sigma_{\alpha})$ and $(\gamma_{\alpha})$  of $n$-bounded bigroup homomorphisms converge to zero uniformly on some zero neighborhoods $U,V\subseteq X$. Let $W$ be an arbitrary
zero neighborhood in $X$. So, there is a zero neighborhood
$W_1$ with $ W_1+W_1 \subseteq W$. There are some $\alpha_0$ and $\alpha_1$ such that $\sigma_{\alpha}(U,V) \subseteq W_1$ for each $\alpha\geq\alpha_0$ and $\gamma_{\alpha}(U,V) \subseteq W_1$ for each $\alpha\geq\alpha_1$. Choose an $\alpha_2$ with $\alpha_2\geq \alpha_0$ and $\alpha_2\geq\alpha_1$. If $\alpha\geq\alpha_2$ then $(\sigma_{\alpha}+\gamma_{\alpha})(U,V) \subseteq \sigma_{\alpha}(U,V)+\gamma_{\alpha}(U,V) \subseteq W_1+W_1 \subseteq W$. Thus, the addition is continuous.
Now, we show the continuity of the module multiplication. Suppose $(r_{\alpha})$ is a net in $R$ which is convergent to zero. There are some neighborhoods $V_1\subseteq R$ and $W_2\subseteq X$ such that $V_1W_2\subseteq W$. Find an $\alpha_3$ with $\sigma_{\alpha}(U,V) \subseteq W_2$ for each $\alpha\geq \alpha_3$. Take an $\alpha_4$ such that $(r_{\alpha})\subseteq V_1$ for each $\alpha\geq \alpha_4$. Choose an $\alpha_5$ with $\alpha_5\geq\alpha_3$ and $\alpha_5\geq \alpha_4$. If $\alpha \geq \alpha_5$ then $r_{\alpha}\sigma_{\alpha}(U,V)\subseteq V_1W_2\subseteq W$, as asserted.
\end{proof}
\begin{theorem}
The operations of addition and module multiplication in $B_{\frac{n}{2}}(X\times X)$ are continuous with respect to the topology of $\sigma$-uniform convergence on some zero neighborhood.
\end{theorem}
\begin{proof}
Suppose two nets $(\sigma_{\alpha})$ and $(\gamma_{\alpha})$  of ${\frac{n}{2}}$-bounded bigroup homomorphisms converge to zero $\sigma$-uniformly on some zero neighborhood $U\subseteq X$. Fix a bounded set $B\subseteq X$. Let $W$ be an arbitrary
zero neighborhood in $X$. So, there is a zero neighborhood
$W_1$ with $ W_1+W_1 \subseteq W$. There are some $\alpha_0$ and $\alpha_1$ such that $\sigma_{\alpha}(U,B) \subseteq W_1$ for each $\alpha\geq\alpha_0$ and $\gamma_{\alpha}(U,B) \subseteq W_1$ for each $\alpha\geq\alpha_1$. Choose an $\alpha_2$ with $\alpha_2\geq \alpha_0$ and $\alpha_2\geq\alpha_1$. If $\alpha\geq\alpha_2$ then $(\sigma_{\alpha}+\gamma_{\alpha})(U,B) \subseteq \sigma_{\alpha}(U,B)+\gamma_{\alpha}(U,B) \subseteq W_1+W_1 \subseteq W$. Thus, the addition is continuous.
Now, we show the continuity of the module multiplication. Suppose $(r_{\alpha})$ is a net in $R$ which is convergent to zero. There are some neighborhoods $V_1\subseteq R$ and $W_2\subseteq X$ such that $V_1W_2\subseteq W$. Find an $\alpha_3$ with $\gamma_{\alpha}(U,B) \subseteq W_2$ for each $\alpha\geq \alpha_3$. Take an $\alpha_4$ such that $(r_{\alpha})\subseteq V_1$ for each $\alpha\geq \alpha_4$. Choose an $\alpha_5$ with $\alpha_5\geq\alpha_3$ and $\alpha_5\geq \alpha_4$. If $\alpha \geq \alpha_5$ then $r_{\alpha}\sigma_{\alpha}(U,B)\subseteq V_1W_2\subseteq W$, as we wanted.
\end{proof}
\begin{theorem}
The operations of addition and module multiplication in $B_{b}(X\times X)$ are continuous with respect to the topology of uniform convergence on bounded sets.
\end{theorem}
\begin{proof}
Suppose two nets $(\sigma_{\alpha})$ and $(\gamma_{\alpha})$  of $b$-bounded bigroup homomorphisms converge to zero uniformly on bounded sets. Fix two bounded sets $B_1,B_2\subseteq X$. Let $W$ be an arbitrary
zero neighborhood in $X$. So, there is a zero neighborhood
$W_1$ with $ W_1+W_1 \subseteq W$. There are some $\alpha_0$ and $\alpha_1$ such that $\sigma_{\alpha}(B_1,B_2) \subseteq W_1$ for each $\alpha\geq\alpha_0$ and $\gamma_{\alpha}(B_1,B_2) \subseteq W_1$ for each $\alpha\geq\alpha_1$. Choose an $\alpha_2$ with $\alpha_2\geq \alpha_0$ and $\alpha_2\geq\alpha_1$. If $\alpha\geq\alpha_2$ then $(\sigma_{\alpha}+\gamma_{\alpha})(B_1,B_2) \subseteq \sigma_{\alpha}(B_1,B_2)+\gamma_{\alpha}(B_1,B_2) \subseteq W_1+W_1 \subseteq W$. Thus, the addition is continuous.
Now, we show the continuity of the module multiplication. Suppose $(r_{\alpha})$ is a net in $R$ which is convergent to zero. There are some neighborhoods $V_1\subseteq R$ and $W_2\subseteq X$ such that $V_1W_2\subseteq W$. Find an $\alpha_3$ with $\gamma_{\alpha}(B_1,B_2) \subseteq W_2$ for each $\alpha\geq \alpha_3$. Take an $\alpha_4$ such that $(r_{\alpha})\subseteq V_1$ for each $\alpha\geq \alpha_4$. Choose an $\alpha_5$ with $\alpha_5\geq\alpha_3$ and $\alpha_5\geq \alpha_4$. If $\alpha \geq \alpha_5$ then $r_{\alpha}\sigma_{\alpha}(B_1,B_2)\subseteq V_1W_2\subseteq W$, as asserted.
\end{proof}
In this step, we investigate whether each class of bounded bigroup homomorphisms is uniformly complete. The answer for $B_b(X\times X)$ is affirmative but for other cases there exist counterexamples.
\begin{remark}\label{100}
The class $B_{n}(X\times X)$ can contain a Cauchy sequence whose limit is not an $n$-bounded bigroup homomorphism.
Let $X={\Bbb R}^{\Bbb N}$, the space of all real sequences, with the coordinate-wise topology and the pointwise product. Define the bigroup homomorphisms $\sigma_n$ on $X$ as follows:
\[\sigma_n(x,y)=(x_1y_1,\ldots,x_ny_n,0,\ldots),\]
in which $x=(x_i)_{i=1}^{\infty}$ and $y=(y_i)_{i=1}^{\infty}$. Each $\sigma_n$ is $n$-bounded. For, if
\[U_n =\{x \in X, |x_j|<1,j=0,1,\ldots,n\},\]
then, $\sigma_n(U_n,U_n)$ is bounded in $X$. Also, $(\sigma_n)$ is a Cauchy sequence in $B_{n}(X\times X)$. Because if $W$ is an arbitrary zero neighborhood in $X$, without loss of generality, we may assume that it is of the form
\[W=(-\varepsilon_1,\varepsilon_1)\times\ldots\times(-\varepsilon_r,\varepsilon_r)\times \Bbb R\times \Bbb R\times\ldots,\]
in which $\varepsilon_i>0$. So, for $m,n >r$; we have $(\sigma_n-\sigma_m)(X,X)\subseteq W$. Also, $(\sigma_n)$ converges uniformly on $(X,X)$ to the bigroup homomorphism $\sigma$ defined by
\[\sigma(x,y)=(x_1y_1,x_2y_2,\ldots).\]
But we have seen in Example \ref{1} that $\sigma$ is not $n$-bounded.
\end{remark}
\begin{remark}
The class $B_{{\frac{n}{2}}}(X\times X)$ can contain a Cauchy sequence whose limit is not an ${\frac{n}{2}}$-bounded bigroup homomorphism.
Let $X$ be $\ell_{\infty}$, with the pointwise product and the uniform norm topology, and $Y$ be $\ell_{\infty}$, with the zero multiplication and the topology induced by norm. Consider bigroup homomorphisms $\sigma_n$ from $X\times Y$ to $X$ as in Remark \ref{100}. It is not difficult to see that each $\sigma_n$ is  ${\frac{n}{2}}$-bounded. Also, $(\sigma_n)$ is a Cauchy sequence in $B_{{\frac{n}{2}}}(X\times X)$ which is convergent $\sigma$-uniformly on $X$ to the bigroup homomorphism $\sigma$ described in Example \ref{200}, so that it is not an  ${\frac{n}{2}}$-bounded bigroup homomorphism.
\end{remark}
\begin{proposition}
Suppose  a net $(\sigma_{\alpha})$ of $b$-bounded bigroup homomorphisms converges to a bigroup homomorphism $\sigma$ uniformly on bounded sets.
Then $\sigma$ is also $b$-bounded.
\end{proposition}
\begin{proof}
Fix  bounded sets $B_1,B_2\subseteq X$. Let $W$ be an arbitrary zero neighborhood in $X$. There is a zero neighborhood $W_1$ such that $W_1+W_1\subseteq W$. Choose a zero neighborhood $V_1\subseteq R$ and a zero neighborhood $W_2\subseteq X$ with $V_1 W_2 \subseteq W_1$. There is an $\alpha_0$ such that $(\sigma_{\alpha}-\sigma)(B_1,B_2) \subseteq W_2$ for each $\alpha\geq\alpha_0$. Fix an $\alpha\geq\alpha_0$. So, there is a zero neighborhood $V_2\subseteq V_1$ with $V_2\sigma_{\alpha}(B_1,B_2) \subseteq W_2$. Therefore,
\[V_2\sigma(B_1,B_2) \subseteq V_2\sigma_{\alpha}(B_1,B_2) + V_2 W_2 \subseteq W_2+V_1W_2 \subseteq W_1+W_1 \subseteq W.\]
\end{proof}
\section{Bounded sets in topological groups}
Let us start with some remarks on boundedness which clarify the context. Suppose $X$ is a topological vector space. When one wants to define a bounded set in $X$, there are two absolutely fruitful tools; scalar multiplication and absorbing neighborhoods at zero. These objects help us to match our intrinsic of boundedness in topological vector spaces; namely, a subset is bounded if it lies in a big enough ball. Now, consider the case when $G$ is a topological group. These two handy material are not available. Of course, it is possible to define bounded sets in a topological group by replacing scalar multiplication with group multiplication in the definition of a bounded set in a topological vector space but this does not meet our intuition of a bounded set since for example the multiplicative group $S^1$ is not bounded in this manner. In addition, it is also possible to consider boundedness in a topological group like totally boundedness in a topological vector space but this one also does not match our intrinsic since it is similar to compactness in the additive group $\Bbb R$. Following \cite{Km}, a subset $B$ in an abelian topological group $G$ is called bounded if for each neighborhood $U$ of the identity, there is an $n\in\Bbb N$ such that $B\subseteq nU$. Let $(G_{\alpha})_{\alpha\in \Lambda}$ be a family of abelian topological groups and $G=\prod_{\alpha\in \Lambda}G_{\alpha}$. It is an easy job to see that $G$ is again an abelian topological group with respect to the both product and box topologies. In this step, we characterise bounded sets of $G$ in terms of bounded sets of $(G_{\alpha})'s$. All topological groups are assumed to be abelian and Hausdorff.

First, we improve \cite[Theorem 1]{Km}.
\begin{theorem}\label{10}
Let $(G_{\alpha})_{\alpha \in\Lambda}$ be a family of abelian topological groups and $G=\prod_{\alpha\in \Lambda}G_{\alpha}$ with the product topology. Then $B\subseteq G$ is bounded if and only if there exists a family of subsets $(B_{\alpha})_{\alpha\in\Lambda}$ such that each $B_{\alpha}\subseteq G_{\alpha}$ is bounded and $B\subseteq \prod_{\alpha\in\Lambda}B_{\alpha}$.
\end{theorem}
\begin{proof}
Suppose $B\subseteq G$ is bounded. Put
\begin{center}
$B_{\alpha}=\{x\in G_{\alpha}: \exists y=(y_{\beta})\in B   $ {and} $   x  $ {is} $ {\alpha} $ {-th coordinate of} $ y\}.$
\end{center}
Each $B_{\alpha}$ is bounded. For, if $U_{\alpha}$ is a neighborhood of identity in $G_{\alpha}$, put
\[U=U_{\alpha}\times\prod_{\beta\neq \alpha}G_{\beta}.\]
Indeed, $U$ is a neighborhood of identity in $G$. Therefore there is a positive integer $n$ with $B\subseteq nU$ so that $B_{\alpha}\subseteq nU_{\alpha}$. Now, it is not difficult to see that $B\subseteq \prod_{\alpha\in\Lambda}B_{\alpha}$.
The converse is a consequence of \cite[Theorem 1]{Km}.
\end{proof}
Note that in a general abelian topological group, every singleton is not necessarily bounded; in other words, not every neighborhood at identity is absorbing. For example, let $G$ be an abelian topological group. Put $H=G\times Z_2$ with the product topology. Then $G\times \{0\}$ is a zero neighborhood which is not absorbing. On the other hand, when $G$ is a connected abelian group, by \cite[Chapter III, Theorem 6]{Taq}, $G$ is absorbed by positive powers of any neighborhood at identity so that singletons will be bounded. Nevertheless, connectedness is a sufficient condition; consider the additive group $\Bbb Q$. In the case when in a topological group $G$, singletons are bounded, compact sets are bounded and therefore we can consider the notion " Heine-Borel" property. Recall that $G$ has the Heine-Borel property if every closed bounded subset of $G$ is compact. Now, we have the following result.
\begin{corollary}
Suppose $(G_{\alpha})_{\alpha\in\Lambda}$ are a family of topological groups in which singletons are bounded and $G=\prod_{\alpha\in \Lambda}G_{\alpha}$ with the product topology. Then singletons are also bounded in $G$.
\end{corollary}
\begin{corollary}
Let $(G_{\alpha})_{\alpha\in\Lambda}$ be a family of topological groups in which singletons are bounded and $G=\prod_{\alpha\in \Lambda}G_{\alpha}$ with the product topology. Then $G$ has the Heine-Borel property if and only if each $G_{\alpha}$ has.
\end{corollary}
\begin{theorem}\label{300}
Let $(G_{\alpha})_{\alpha\in\Lambda}$ be a family of abelian topological groups and $G=\prod_{\alpha\in\Lambda}G_{\alpha}$ with the box topology. Then $B\subseteq G$ is bounded if and only if there exists a finite set $\{\alpha_1,\ldots, \alpha_k\}$ of indices such that  $B\subseteq B_{{\alpha}_1}\times\ldots\times B_{{\alpha}_k}\times\prod_{\beta\in \Lambda-\{\alpha_1,\ldots,\alpha_k\}}\{e_{\beta}\}$, where $e_{\beta}$ denotes the identity element of $G_{\beta}$.
\end{theorem}
\begin{proof}
Suppose $B\subseteq G$ is bounded and there is a net $(c_{\alpha})$ of non-identity elements of $(G_{\alpha})$ such that each $c_{\alpha}$ belongs to a component of $B$. There is a neighborhood $U_{\alpha}$ at identity element $e_{\alpha}$ in $G_{\alpha}$ such that $c_{\alpha} \not\in U_{\alpha}$. Partition the index set to a countable collection $(A_n)$. For each $\alpha \in A_n$, take a neighborhood $V_{\alpha}$ at identity with $nV_{\alpha}\subseteq U_{\alpha}$. Put
\[B_1=\prod_{\alpha\in \Lambda} \{c_{\alpha}\}.\]
Obviously, $B_1$ should be bounded in $G$. Now, suppose $V$ is a neighborhood at identity of the form
\[V=\prod_{\alpha\in\Lambda}V_{\alpha}.\]
It is not a difficult job to see that there is no $M>0$ with $B_1\subseteq M V$; for, in this case, $c_{\alpha}\in M V_{\alpha}\subseteq nV_{\alpha}\subseteq U_{\alpha}$, a contradiction. Therefore, for all but finitely many components, $B$ should have identity elements.
Also, by a similar argument that we had in the first direction of the proof of Theorem \ref{10}, we conclude that for some $\{\alpha_1,\ldots,\alpha_k\}$, $B\subseteq B_{{\alpha}_1}\times\ldots\times B_{{\alpha}_k}\times \prod_{\beta\in \Lambda-\{\alpha_1,\ldots,\alpha_k\}}\{e_{\beta}\}$, in which, each $B_{{\alpha}_i}$ is bounded in $G_{{\alpha}_i}$. The other direction is trivial.
\end{proof}
\begin{remark}
Considering the proof of Theorem \ref{300}, we conclude that in the box topology, singletons are never bounded so that in such spaces compact sets are not bounded in general.
\end{remark}
\begin{corollary}
Let $(G_{\alpha})$ be a family of abelian topological groups and $G=\prod_{\alpha\in \Lambda}G_{\alpha}$ with the box topology. Then singletons are not bounded, in general. In particular, $G$ is never connected by \cite[Chapter III, Theorem 6]{Taq} even when all of $G_{\alpha}'s$ are connected. Nevertheless, consider this point that by \cite[Chapter III, Exercise 8]{Taq}, product topology preserves connectedness.
\end{corollary}

\end{document}